\documentclass[11pt,reqno,tbtags]{amsart}
\usepackage{amssymb}
\usepackage{natbib}
\bibpunct[, ]{[}{]}{;}{n}{,}{,}

\allowdisplaybreaks

\newtheorem{theorem}{Theorem}%[section]
\newtheorem{lemma}[theorem]{Lemma}

\theoremstyle{definition}
\newtheorem{example}[theorem]{Example}

\newtheorem{remark}{Remark}
\newtheorem*{warning}{Warning}

\theoremstyle{remark}

\newenvironment{romenumerate}{\begin{enumerate}% gives (i), (ii) etc.
 }{\end{enumerate}}

\newcounter{oldenumi}
\newenvironment{defq}% continues numbering from previous romenumerate
{%\setcounter{oldenumi}{\value{enumi}}
\begin{enumerate} \setcounter{enumi}{\value{oldenumi}}%
}
{\setcounter{oldenumi}{\value{enumi}}\end{enumerate}}
\newcounter{thmenumerate}

\newcounter{xenumerate}   %no left indentation; thus wider lines

\newcommand{\refL}[1]{Lemma~\ref{#1}}
\newcommand{\refR}[1]{Remark~\ref{#1}}
\newcommand{\refS}[1]{Section~\ref{#1}}

\newcommand{\refD}[1]{Definition~\ref{#1}}
\newcommand{\refE}[1]{Example~\ref{#1}}

\newcommand\marginal[1]{\marginpar{\raggedright\parindent=0pt\tiny #1}}

\begingroup
  \divide\count255 by 60
  \multiply\count255 by -60
  \advance\count255 by \time
  \ifnum \count255 < 10 \xdef\klockan{\the\count1.0\the\count255}
  \else\xdef\klockan{\the\count1.\the\count255}\fi
\endgroup

\newcommand\nopf{\qed}   % for theorem without proof
\newcommand\set[1]{\ensuremath{\{#1\}}}

\newcommand\bigpar[1]{\bigl(#1\bigr)}

\def\rompar(#1){\textup(#1\textup)}    % usage: \rompar(...)

\def\xexp(#1){e^{#1}}

\newcommand\ntoo{\ensuremath{{n\to\infty}}}

\newcommand\norml[2]{\|#1\|_{L^{#2}}}
\newcommand\normlp[1]{\|#1\|_{L^{p}}}

\newcommand\ie{i.e.\spacefactor=1000}
\newcommand\eg{e.g.\spacefactor=1000}

\newcommand{\as}{a.s.\spacefactor=1000}

\newcommand\whp{{w.h.p.\spacefactor=1000}}

\newcommand{\tend}{\longrightarrow}
\newcommand\dto{\overset{\mathrm{d}}{\tend}}
\newcommand\pto{\overset{\mathrm{p}}{\tend}}
\newcommand\asto{\overset{\mathrm{a.s.}}{\tend}}

\newcommand\op{o_{\mathrm p}}
\newcommand\Op{O_{\mathrm p}}
\newcommand\Olx[1]{O_{L^{#1}}}
\newcommand\olx[1]{o_{L^{#1}}}
\newcommand\Olp{\Olx p}
\newcommand\olp{\olx p}
\newcommand\Oloo{\Olx\infty}
\newcommand\oloo{\olx\infty}

\newcounter{CC}
\newcounter{cc}
\newcommand\E{\operatorname{\mathbb E{}}}
\renewcommand\P{\operatorname{\mathbb P{}}}
\newcommand\Var{\operatorname{Var}}

\newcommand\gd{\delta}

\newcommand\go{\omega}
\newcommand\gO{\Omega}

\newcommand\eps{\varepsilon}

\newcommand\cE{\mathcal E}

\def\[#1]{[\![#1]\!]}

\newcommand\qq{^{1/2}}

\newcommand\qqw{^{-1/2}}

\newcommand\qw{^{-1}}

\renewcommand{\=}{:=}

\newcommand\xnOanwhp{\ensuremath{X_n=O(a_n) \text{ \whp}}}
\newcommand\xnoanwhp{\ensuremath{X_n=o(a_n) \text{ \whp}}}

\newcommand{\Holder}{H\"older}

\newcommand\REM[1]{{\raggedright\texttt{[#1]}\par\marginal{XXX}}}

\newcommand\urladdrx[1]{{\urladdr{\def~{{\tiny$\sim$}}#1}}}

\begin{document}
\title%[]
{Probability asymptotics: notes on notation}

\date{24 April 2009; typos corrected 30 July 2009} % (typeset \today{} \klockan)} %; revised ...
\thanks{These notes were written at Institut Mittag-Leffler,
Djursholm, Sweden, during the programme ``Discrete Probability'' 2009.
I thank several participants for helpful comments and suggestions.}

\author{Svante Janson}
\address{Department of Mathematics, Uppsala University, PO Box 480,
SE-751~06 Uppsala, Sweden}
\email{svante.janson@math.uu.se}
\urladdrx{http://www.math.uu.se/~svante/}

%\keywords{<keywords>}
\subjclass[2000]{} 
%{60C05 (68P10,68W40)} %%{Primary: <subject>; Secondary: <subject>}

\begin{abstract} 
We define and compare several different versions of the $O$ and $o$
notations for random variables.
The main purpose is to give proper definitions
in order to avoid ambiguities and mistakes.
\end{abstract}

\maketitle

\section{Introduction}\label{S:intro}

There are many situations where one studies asymptotics of random
variables or events, and it is therefore important to have good
definitions and notations for random asymptotic
properties. Probabilists use often the standard concepts
\emph{convergence almost surely}
($\asto$),
\emph{convergence in probability} ($\pto$)
and \emph{convergence in distribution} ($\dto$);
see any textbook in probability theory for definitions. (Two of my
favorite references, at different levels, are \citet{Gut} and
\citet{Kallenberg}.) 

Other notations, often used in, for example, discrete probability such
as probabilistic combinatorics, are
probabilistic versions of the $O$ and 
$o$ notation. These notations are very useful; however, several
versions exist with somewhat different definitions (some equivalent
and some not), so some care is needed when using them.
In particular, I have for many years avoided the notations 
``$O(\cdot)$ \whp'' and ``$o(\cdot)$ \whp'' on the
grounds that these combine two different asymptotic notations in an
ambiguous and potentially dangerous way. (In which order do the
quantifiers really come in a formal definition?) I now have changed
opinion, and I regard these as valid and useful notations, provided
proper definitions are given. One of the purposes of these notes is to
state such definitions explicitly (according to my interpretations of
the notions; I hope that others interpret them in the same way).
Moreover, various relations and equivalences between different notions 
are given.

The results below are all elementary and more or less well-known. I do
not think that any results are new, and they are in any case
at the level of exercises in probability theory rather than advanced theorems.
Nevertheless, I hope that this collection of various definitions and
relations may be useful to myself and to others that use these 
concepts. (See also the similar discussion in \cite[Section 1.2]{JLR} 
of many of these, and some further, notions.)

%\section{Some standard asymptotic notions}

We suppose throughout that $X_n$ are random variables and
$a_n$ positive numbers, $n=1,2,\dots$; unless we say otherwise, the
$X_n$ do not have to be defined on the same probability
space. (In other words, only their distributions matter.)
All unspecified limits are as \ntoo.

All properties below relating $X_n$ and $a_n$ depend only on
$X_n/a_n$; we could thus normalize and assume that
$a_n=1$, but for convenience in applications,
we will state the results in the more general form
with arbitrary positive $a_n$.

\section{$O$ and $o$}

We begin with the standard definitions for non-random sequences.
Assume that $b_n$ is some sequence of numbers.

%\begin{definition}
  \begin{defq}
\item\label{DO}
  $b_n=O(a_n)$ if there exist constants $C$ and
  $n_0$ such that $|b_n|\le C a_n$ for $n\ge n_0$.
Equivalently, 
\begin{equation}\label{O}
b_n=O(a_n)\iff \limsup_\ntoo \frac{|b_n|}{a_n} <\infty.  
\end{equation}
\item\label{Do}
  $b_n=o(a_n)$ if $b_n/a_n\to0$.
Equivalently,
$b_n=o(a_n)$ if for every $\eps>0$
there exists $n_\eps$ such that $|b_n|\le \eps a_n$ for $n\ge n_\eps$.
  \end{defq}
%\end{definition}

\begin{remark}
When considering sequences as here,
  the qualifier ``$n\ge n_0$'' is not really necessary in the
  definition of $O(\cdot)$, and it is often omitted, which is
  equivalent to replacing $\limsup$ by $\sup$ in \eqref{O}. 
The only effect of using an $n_0$ is to allow us to have $a_n$ or $b_n$
  undefined or infinite, or $a_n=0$, for some small $n$;
  for example, we may write $O(\log n)$ without making an
  explicit exception for $n=1$. Indeed, if everything is well defined
  and $a_n>0$,
  as we assume in these notes,
and $|b_n|\le C a_n$ for $n\ge n_0$, then
$\sup_n|b_n/a_n| \le\max(C, \max_{i\le n_0}
  |b_n/a_n|)<\infty$.

On the other hand, when considering functions of a continuous
  variable, the two versions of $O(\cdot)$ are different and
  should be distinguished. (Both versions are used in the literature.)
For example, there is a difference between the conditions
$f(x)=O(x)$ on $(0,1)$ (meaning $\sup_{0<x<1}|f(x)/x|<\infty$, 
\ie, a uniform estimate on $(0,1)$), and  
$f(x)=O(x)$ as $x\to0$ (meaning $\limsup_{x\to0}|f(x)/x|<\infty$, 
\ie, an asymptotic estimate for small $x$); the former but
not the latter entails that $f$ is bounded also close to 1.
(As shown here, when necessary, the two versions of $O$ can be
  distinguished by adding qualifiers such as ``\ntoo'' or
  ``$x\to0$''  for the asymptotic version and ``$n\ge1$'' or
  ``$x\in(0,1)$'' for the uniform version. Often, however, such
  qualifiers are omitted when the meaning is clear from the context.)
\end{remark}

\section{Convergence in probability}

The standard definition of convergence in probability is as follows.
\begin{defq}
%\begin{definition}
\item\label{Dpto}
  $X_n\pto 0$ if for every $\eps>0$,
  $\P(|X_n|>\eps)\to0$.
Equivalently,
$$X_n\pto0\iff\sup_{\eps>0}\limsup_{\ntoo}\P(|X_n|>\eps)=0.$$
\end{defq}
%\end{definition}

\begin{remark}
More generally, one defines $X_n\pto a$ for a constant
$a$ similarly, or by $X_n\pto a\iff X_n-a\pto0$.
If the random variables $X_n$ are defined on the same probability
space, one further defines $X_n\pto X$ for a random variable
$X$ (defined on that probability space) by
$X_n\pto X$ if $X_n-X\pto0$. 
\end{remark}

It is well-known that convergence in probability to a constant is
equivalent to convergence in distribution to the same constant.
(See \eg{} \cite{Billingsley, Gut, Kallenberg} for
definition and equivalent charaterizations of convergence in distribution.)
In particular,
\begin{equation}\label{ptodto}
  X_n\pto 0 \iff X_n\dto 0.
\end{equation}

\section{With high probability}

For events, we are in particular interested in typical events, \ie,
events that occur with probability tending to 1 as \ntoo. Thus, we
consider an event $\cE_n$ for each $n$, and 
say that:
\begin{defq}
  \item\label{Dwhp}
$\cE_n$ holds \emph{with high probability} (\whp) if
$\P(\cE_n)\to1$ as \ntoo. 
\end{defq}
This too is a common and useful notation. 

\begin{remark}\label{Ras}
A common name in probabilistic combinatorics for this property
has been
``almost surely'' or ``a.s.'', but that conflicts with the well
established use of this phrase (and abbreviation) in probability
theory where it means probability \emph{equal} to 1. In my opinion, 
``almost surely'' (\as)
should be reserved for its probabilistic meaning, since
giving it a different meaning might lead to confusion.
(In these notes, \as{} is used in the standard sense.)
Another alternative name for \ref{Dwhp} is ``asymptotically almost surely'' or
``a.a.s.''. This name is commonly used, for example in \cite{JLR}, 
and the choice between the synonymous ``\whp'' (often written whp) and
``a.a.s.'' is a 
matter of taste. (At present, I prefer \whp, so I use it here.)
\end{remark}

\refD{Dpto} of convergence in probability can be stated using
\whp{} as:
\begin{equation}\label{ptowhp}
  X_n\pto 0 \iff 
\text{ for every $\eps>0$, } |X_n|\le \eps \text{ \whp}
\end{equation}

\section{$\Op$ and $\op$}

A probabilistic version of $O$ that is frequently used is the following:
\begin{defq}
  \item\label{DOp}
$X_n=\Op(a_n)$ if for every $\eps>0$ there exists
constants $C_\eps$ and $n_\eps$ such that
$\P(|X_n|\le C_\eps a_n)>1-\eps$ for every $n\ge n_\eps$.
\end{defq}
In other words, $X_n/a_n$ is bounded, up to an exceptional event of
arbitrarily small (but fixed) positive probability.
This is also known as $X_n/a_n$ being \emph{bounded in probability}.

The definition \ref{DOp}
 can be rewritten in equivalent forms, for example as follows.
\begin{lemma}\label{LOp}
  The following are equivalent:
  \begin{romenumerate}
	\item \label{LOp1}
$X_n=\Op(a_n)$.
\item\label{LOp2}
For every $\eps>0$ there exists 
 $C_\eps$  such that
$\P(|X_n|\le C_\eps a_n)>1-\eps$ for every $n$.
\item \label{LOplimsup}
For every $\eps>0$ there exists $C_\eps$ such that
$%\begin{equation*}
\limsup_\ntoo\P(|X_n|> C_\eps a_n)<\eps.
$%\end{equation*}
\item \label{LOpsup}
For every $\eps>0$ there exists $C_\eps$ such that
$%\begin{equation*}
\sup_n\P(|X_n|> C_\eps a_n)<\eps.
$%\end{equation*}
\item \label{LOplimlimsup}
$%\begin{equation*}
\lim_{C\to\infty}\limsup_\ntoo\P(|X_n|> C a_n)=0.
$%\end{equation*}
\item \label{LOplim+sup}
$%\begin{equation*}
\lim_{C\to\infty}\sup_n\P(|X_n|> C a_n)=0.
$%\end{equation*}
  \end{romenumerate}
\end{lemma}

\begin{proof}
  \ref{LOp1}$\implies$\ref{LOp2} follows by
  increasing $C_\eps$ in \ref{DOp} such that
  $\P(|X_n|\le C_na_n)>1-\eps$ for
  $n=1,\dots,n_0$ too.

%The converse
\ref{LOp2}$\implies$\ref{LOp1} is trivial.

\ref{LOp1}$\iff$\ref{LOplimsup}$\iff$\ref{LOplimlimsup}
  and
\ref{LOp2}$\iff$\ref{LOpsup}$\iff$\ref{LOplim+sup}
 are easy and left to the reader. 
\end{proof}

\begin{remark}\label{Rtight}
Another term equivalent to ``bounded in probability'' is
\emph{tight};
thus, $X_n=\Op(a_n)$ if and only if the family
\set{X_n/a_n} is tight.
By Prohorov's theorem \cite{Billingsley,Kallenberg},
tightness is equivalent to relative compactness of the set of
distributions. Hence,
$X_n=\Op(a_n)$ 
if and only if every subsequence of $X_n/a_n$ has a
subsequence that converges in distribution; however, different
convergent subsequences may have different limits.
In particular, if $X_n/a_n$ converges in distribution, then $X_n=\Op(a_n)$.
\end{remark}

The corresponding $\op$ notation can be defined as follows.

\begin{defq}
  \item\label{Dop}
$X_n=\op(a_n)$ if for every $\eps>0$ there exists
$n_\eps$ such that\\
$\P(|X_n|\le \eps a_n)>1-\eps$ for every $n\ge n_\eps$.
\end{defq}

The definition \ref{Dop} too has several equivalent forms, for example as follows.
\begin{lemma}\label{Lop}
  The following are equivalent:
  \begin{romenumerate}
	\item \label{Lop1}
$X_n=\op(a_n)$.
\item \label{Loplim}
For every $\eps>0$,
$%\begin{equation*}
\P(|X_n|> \eps a_n)\to0.
$%\end{equation*}
\item \label{Lopsuplimsup}
$%\begin{equation*}
\sup_{\eps>0}\limsup_\ntoo\P(|X_n|> \eps a_n)=0.
$%\end{equation*}
\item \label{Lopwhp}
For every $\eps>0$,
$|X_n|\le \eps a_n$ \whp
\item \label{Loppto}
$X_n/a_n\pto0$.
  \end{romenumerate}
\end{lemma}

\begin{proof}
  \ref{Lop1}$\iff$\ref{Loplim} follows by standard arguments which
  we omit.

%The converse
\ref{Loplim}$\iff$\ref{Lopwhp} is immediate by the definition
\ref{Dwhp} of \whp

\ref{Loplim}$\iff$\ref{Loppto} is immediate by the definition
\ref{Dpto} of $\pto$. 

(Further, 
\ref{Lopwhp}$\iff$\ref{Loppto} follows by \eqref{ptowhp}.)
\end{proof}

\section{Using arbitrary functions $\go(n)$}

Some papers use properties that are stated using
an arbitrary function
(or sequence)
$\go(n)\to\infty$. 
(Or,  equivalently, stated in
terms of an arbitrary sequence $\gd_n\=1/\go(n)\to0$; see for
example \cite[Lemma 4.9]{Kallenberg}, which is essentially the same as
\ref{LOpgo1}$\iff$\ref{LOpgopto} in 
the following lemma.)
They are equivalent to $\Op$ or $\op$ by the
following lemmas.
(I find the $\Op$ and $\op$ notation more
transparent and prefer it to using $\go(n)$.)

\begin{lemma}\label{LOpgo}
  The following are equivalent:
  \begin{romenumerate}
	\item \label{LOpgo1}
$X_n=\Op(a_n)$.
\item\label{LOpgowhp}
For every function $\go(n)\to\infty$,
$|X_n|\le\go(n)a_n$ \whp
\item\label{LOpgopto}
For every function $\go(n)\to\infty$,
$|X_n|/(\go(n)a_n)\pto 0$.
  \end{romenumerate}
\end{lemma}

\begin{proof}
\ref{LOpgo1}$\implies$\ref{LOpgowhp}.
For every $\eps>0$, choose $C_\eps$ as in \refL{LOp}\ref{LOplimsup}.
Then $\go(n)>C_\eps$ for large $n$, and thus
$$\limsup_\ntoo\P(|X_n|> \go(n) a_n)\le
\limsup_\ntoo\P(|X_n|> C_\eps a_n)<\eps.$$
Hence, $\limsup_\ntoo\P(|X_n|> \go(n) a_n)=0$, which is \ref{LOpgowhp}.

\ref{LOpgowhp}$\implies$\ref{LOpgo1}.
If $X_n=\Op(a_n)$ does not hold, then, by the definition
\ref{DOp},  there exists $\eps>0$ such that for every
$C$ there exist arbitrarily large $n$ with
$\P(|X_n|>Ca_n)\ge\eps$. We may thus inductively define
an increasing sequence $n_k$, $k=1,2,\dots$, such that
$\P(|X_{n_k}|>ka_{n_k})\ge\eps$. Define
$\go(n)$ by $\go(n_k)=k$ and $\go(n)=n$ for
$n\notin\set{n_k}$. Then $\go(n)\to\infty$ and
$\P(|X_{n}|>\go(n)a_{n})\not\to0$, so
\ref{LOpgowhp} does not hold.

\ref{LOpgowhp}$\implies$\ref{LOpgopto}.
If $\eps>0$, then $\eps\go(n)\to\infty$
too, and thus by \ref{LOpgowhp}
$|X_n|\le\eps\go(n)a_n$ \whp{}
Thus $X_n/(\go(n)a_n)\pto0$ by \eqref{ptowhp}.

\ref{LOpgopto}$\implies$\ref{LOpgowhp}.
Take $\eps=1$ in \eqref{ptowhp}.
\end{proof}

\begin{remark}\label{Rbgo}
 \refL{LOpgo} generalizes the corresponding result for a non-random sequence
  \set{b_n}:  
%(which is a special case):
\set{b_n} is bounded $\iff$
  $|b_n|\le\go(n)$ for every $\go(n)\to\infty$
$\iff$
  $|b_n|/\go(n)\to0$ for every $\go(n)\to\infty$.
\end{remark}

\begin{lemma}\label{Lopgo}
  The following are equivalent:
  \begin{romenumerate}
	\item \label{Lopgo1}
$X_n=\op(a_n)$.
\item\label{Lopgowhp}
For some function $\go(n)\to\infty$,
$|X_n|\le a_n/\go(n)$ \whp
\item\label{Lopgopto}
For some function $\go(n)\to\infty$,
$\go(n)|X_n|/a_n\pto 0$.
  \end{romenumerate}
\end{lemma}

\begin{proof}
\ref{Lopgo1}$\implies$\ref{Lopgowhp}.
By the definition \ref{Dop}, for every $k$ there exists
$n_k$ such that if $n\ge n_k$, then
$\P(|X_n|>k\qw a_n)<k\qw$.
We may further assume that $n_k>n_{k-1}$, with $n_0=1$.
Define $\go(n)=k$ for $n_k\le n<n_{k+1}$.
Then $\go(n)\to\infty$ and 
$\P(|X_n|>\go(n)\qw a_n)<\go(n)\qw$. Since $\go(n)\qw\to 0$, this yields
\ref{Lopgowhp}.

\ref{Lopgowhp}$\implies$\ref{Lopgo1}.
Let $\eps>0$. Since $\go(n)\qw\le\eps$ for
large $n$, \ref{Lopgowhp} implies that
$|X_n|\le\eps a_n$ \whp{}
Thus
$X_n=\op(a_n)$ by \refL{Lop}.

\ref{Lopgowhp}$\implies$\ref{Lopgopto}.
If $\go(n)$ is as in \ref{LOpgowhp}, then 
$\go(n)\qq|X_n|/a_n\le\go(n)\qqw$ \whp{}; since
$\go(n)\qqw\to0$, this implies
$\go(n)\qq X_n/a_n\pto0$, so \ref{Lopgopto} holds with the function 
$\go(n)\qq\to\infty$.

\ref{Lopgopto}$\implies$\ref{Lopgowhp}.
Take $\eps=1$ in \eqref{ptowhp}.
\end{proof}

\section{$\Olp$ and $\olp$}\label{Solp}

The following notations are less common but sometimes very useful. Recall
that for $0<p<\infty$ the $L^p$ norm of a random
variable $X$ is $\normlp{X}\=\bigpar{\E|X|^p}^{1/p}$.
Let $p>0$ be a fixed number. (In applications, usually
$p=1$ or $p=2$.)

  \begin{defq}
\item\label{DOl}
  $X_n=\Olp(a_n)$ if $\normlp{X_n}=O(a_n)$.
\item\label{Dol}
  $X_n=\olp(a_n)$ if $\normlp{X_n}=o(a_n)$.
  \end{defq}
In other words, 
  $X_n=\Olp(a_n)\iff \E|X_n|^p=O(a_n^p)$ and 
  $X_n=\olp(a_n)\iff \E|X_n|^p=o(a_n^p)$;
in particular,
  $X_n=\Olx1(a_n)\iff \E|X_n|=O(a_n)$ and 
  $X_n=\olx1(a_n)\iff \E|X_n|=o(a_n)$.

$X_n=\olx1(a_n)$ thus says that
$\E|X_n/a_n|\to0$, which often is expressed as
\emph{$X_n/a_n\to0$ in mean}.
More generally,
$X_n=\olp(a_n)$ is the same as
$\E|X_n/a_n|^p\to0$, which is called
\emph{$X_n/a_n\to0$ in $p$-mean} (or in $L^p$).
(For $p=2$, a common name is
\emph{$X_n/a_n\to0$ in square mean}.)

  We may also take $p=\infty$. Since $L^\infty$ is
  the space of bounded random variables and
  $\norml{X}{\infty}$ is the essential
  supremum of $|X|$, \ie, 
$\norml{X}\infty\=\inf\set{C:|X|\le
  C\text{ a.s.}}$, 
the definitions \ref{DOl}--\ref{Dol} can for
  $p=\infty$ be rewritten as:
\begin{defq}
\item\label{DOloo}
$X_n=\Oloo(a_n)$ if there exists a constant
  $C$ such that $|X_n|\le Ca_n$ \as
\item\label{Doloo}
$X_n=\oloo(a_n)$ if there exists a 
sequence $\gd_n\to0$ such  
that $|X_n|\le \gd_na_n$ \as
\end{defq}

\begin{remark}\label{Rdiscrete}
  In applications in discrete probability, typically each $X_n$
  is a discrete random variable taking only a finite number of
  possible values, each with positive probability. In such cases
(and more generally if the number of values is countable, each with
  positive probability),
  $|X_n|\le C a_n$ \as{} $\iff|X_n|\le C a_n$ 
surely (\ie, for each realization), and 
  $|X_n|\le \gd_n a_n$ \as{} $\iff|X_n|\le \gd_n a_n$ surely.
\end{remark}

The notions $\Olp$ and $\olp$ are useful for example when considering sums of a
growing (or infinite) number of terms, since (for $p\ge1$) such
estimates can be added by Minkowski's inequality.
For example, if $X_n=\sum_{i=1}^n Y_{ni}$, and
$Y_{ni}=\Olp(a_n)$ (uniformly in $i$) for some
$p\ge1$, then $X_n=\Olp(na_n)$, and similarly for
$\olp$.
Note that the corresponding statement for $\Op$ and
$\op$ are false.
(Example: Let $Y_{ni}$ be independent with 
$\P(Y_{ni}=n^2)=1-\P(Y_{ni}=0)=1/n$ and let $a_n=1$.)

By Lyapunov's (or \Holder's) inequality, 
$X_n=\Olp(a_n)\implies X_n=\Olx q(a_n)$ and
$X_n=\olp(a_n)\implies X_n=\olx q(a_n)$ when
$0<q\le p\le\infty$. Thus the estimates become stronger as $p$ increases.
They are, for all $p$, stronger than $\Op$ and $\op$.

\begin{lemma}
  Let $0<p\le\infty$. Then
$X_n=\Olp(a_n)\implies  X_n=\Op(a_n)$ and 
$X_n=\olp(a_n)\implies  X_n=\op(a_n)$.
\end{lemma}
\begin{proof}
  Immediate from Markov's inequality.
\end{proof}

The converse fails for every $p>0$. (Example for any $p>0$: Take $X_n$ with
$\P(X_{n}=e^n)=1-\P(X_{n}=0)=1/n$ and let $a_n=1$.)

\begin{remark}
For $p<\infty$,
  $X_n=\olp(a_n)$ is equivalent to $X_n=\op(a_n)$
  together with the condition that \set{|X_n/a_n|^p} are
  uniformly integrable, see \eg{} \cite{Gut} or \cite{Kallenberg}.
%\cite[Theorems 5.4.2 and 5.5.9]{Gut} or \cite[Proposition 4.12]{Kallenberg}.
\end{remark}

Another advantage of $\Olp$ and $\olp$ is that they
are strong enough to imply moment estimates:
\begin{lemma}
  \label{Lolmoments}
If $k$ is a positive integer with $k\le p$, then
$X_n=\Olp(a_n)\implies \E X_n^k=O(a_n^k)$ and 
$X_n=\olp(a_n)\implies \E X_n^k=o(a_n^k)$.
\nopf
\end{lemma}

In particular,
$X_n=\Olx1(a_n)\implies \E X_n=O(a_n)$ and 
$X_n=\olx1(a_n)\implies \E X_n=o(a_n)$; further,
$X_n=\Olx2(a_n)\implies \Var X_n=O(a_n^2)$ and 
$X_n=\olx2(a_n)\implies \Var X_n=o(a_n^2)$.

\section{$O$ \whp{} and $o$ \whp}\label{Sowhp}

Since the basic meaning of $O$ is ``bounded by some fixed but
unknown constant'', my interpretation of ``$O(a_n)$ \whp'' is
the following:
\begin{defq}
\item\label{DOwhp}
\xnOanwhp{} if there exists a constant
  $C$ such that $|X_n|\le Ca_n$ \whp  
\end{defq}
Comparing Definitions \ref{DOp} and \ref{DOwhp}, we
see that the latter is a stronger notion:
\begin{equation}
  \xnOanwhp \implies X_n=\Op(a_n),
\end{equation}
but the converse does not hold. (In fact, \ref{DOwhp} is the
same as \ref{DOp} with the restriction that $C_\eps$
must be chosen independent of $\eps$.)
For example, if $X_n/a_n\dto Y$ for some random variable
$Y$, then always $X_n=\Op(a_n)$, see
\refR{Rtight}, but it is easily seen that
$X_n=O(a_n)$ \whp{} if and only if $Y$ is
bounded, \ie, $|Y|\le C$ (\as{}) for some
constant $C<\infty$. (In particular, if $X_n=X$ 
does not depend on $n$,
then always $X_n=\Op(1)$, but $X_n=O(1)$
\whp{} only if $X$ is bounded.)
This also shows that $X_n=\Olp(a_n)$ in general does not
imply $X_n=O(a_n)$ \whp

\begin{remark}
  \label{RtightC}
More generally, $X_n=O(a_n)$ \whp{} if and only if
every subsequence of $X_n/a_n$ 
has a subsequence that converges in distribution to
a bounded random variable, with some uniform bound for all subsequence limits.
\end{remark}

\begin{remark}
  The property $X_n=O(a_n)$ \whp{} was denoted
  $X_n=O_C(a_n)$ in \cite{JLR}. (A notation that perhaps
  was not very successful.)
\end{remark}

Similarly, the  basic meaning of $o$ is ``bounded by some fixed but
unknown sequence $\gd_n\to0$''; thus my interpretation of ``$o(a_n)$ \whp'' is
the following:
\begin{defq}
\item\label{Dowhp}
\xnoanwhp{} if there exists a sequence $\gd_n\to0$ such  
that $|X_n|\le \gd_na_n$ \whp  
\end{defq}

This condition is the same as
\refL{Lopgo}\ref{Lopgowhp} (with
$\gd_n=\go(n)\qw$), and thus \refL{Lopgo}
implies the following equivalence:
\begin{lemma}
  \label{Lowhp}
$\xnoanwhp \iff X_n=\op(a_n)$.
\nopf
\end{lemma}

It is obvious from the definitions \ref{DOwhp} and
\ref{Dowhp} that $o(a_n)$ \whp{} implies
$O(a_n)$ \whp, and we thus have the chain of implications
(where the last two are not reversible):
\begin{equation}\label{oooowhp}
  \op(a_n)
\iff
  o(a_n) \text{ \whp}
\implies
  O(a_n) \text{ \whp}
\implies
  \Op(a_n).
\end{equation}

\begin{warning}
I do not think that  definition \ref{DOwhp} is the only
interpretation of ``$O(a_n)$ \whp'' that is used, so extreme
care is needed when using or seeing this notation to avoid confusion
and mistakes. (For example, I've heard the interpretation that 
``$O(a_n)$ \whp'' should be equivalent to ``$\Op(a_n)$''.)
The risks with    ``$\op$ \whp'' seem smaller; at least, 
I do not know any other reasonable (non-equivalent) interpretation of
it.
\end{warning}

\section{$O$ and $o$ a.s.}\label{SOas}

In this section we assume that the random variables $X_n$ are
defined together on the same probability space $\gO$.
In other words, the variables $X_n$ are coupled.
(In combinatorial situations this is usually \emph{not} the
case, since typically each $X_n$ is defined separately on some
model of ``size'' $n$; however, it happens, for example in a
model that grows in size by some random process.)
This assumption makes it
possible to talk about convergence and other properties \as,
\ie, pointwise (= pathwise) for all points in the probability space $\gO$
except for a subset with probability 0. This means that we consider the
sequence $X_n(\go)$ of real numbers separately for each
point $\go$ in the probability space. Hence, we
apply definitions \ref{DO} and \ref{Do} for
non-random sequences and obtain the following definitions.

\begin{defq}
\item\label{DOas}
  $X_n=O(a_n)$ \as{} if for almost every
  $\go\in\gO$, there exists a number $C(\go)$
  such that $|X_n(\go)|\le C(\go) a_n$.
In other words, $X_n=O(a_n)$ \as{} if there exists a
  \emph{random variable} $C$ such that $|X_n|\le C a_n$ a.s.
Equivalently,
\begin{equation}
    X_n=O(a_n) \text{ \as} \iff 
\limsup_\ntoo \frac{|X_n|}{a_n}<\infty \text{ a.s.}
\end{equation}
\item\label{Doas}
  $X_n=o(a_n)$ \as{} if for almost every
  $\go\in\gO$, $|X_n(\go)|/a_n\to0$.
In other words,   $X_n=o(a_n)$ \as{} if $X_n/a_n\asto0$.
\end{defq}

It is well-known that 
convergence almost surely implies convergence in probability
%$\asto$ implies $\pto$ 
(but not conversely).
Consequently, by \ref{Doas} and Lemmas \ref{Lop} and \ref{Lowhp},
\begin{equation}\label{oooas}
  X_n=o(a_n) \text{ a.s.} \implies X_n=\op(a_n)
\iff X_n=o(a_n) \text{ \whp}
\end{equation}

The situation for $O$ is more complicated.
We first observe the implication
\begin{equation}\label{OOas}
  X_n=O(a_n) \text{ a.s.} \implies X_n=\Op(a_n).
\end{equation}
(The converse does not hold, see \refE{E2} below.)
Indeed, 
if $c$ is any constant, and $C$ is a random variable with
$|X_n|\le Ca_n$ as in
\ref{DOas}, then $\P(|X_n|>ca_n)\le \P(C>c)$,
and thus \refL{LOp}\ref{LOplim+sup} holds because
$\P(C>c)\to0$ as $c\to\infty$. Hence,
\refL{LOp} yields \eqref{OOas}.

However, the following two examples show that neither of 
  $X_n=O(1)$ \as{} and  $X_n=O(1)$ \whp{} implies the other.

\begin{example}\label{E1}
  Let $X_n=X$ be independent of $n$ and let
  $a_n=1$. Then $X_n=O(1)$ \as{} for every random
  variable $X$ (take $C=X$ in \ref{DOas}), but
  $X_n=O(1)$ \whp{} only if $X$ is a bounded
  random variable (\ie, $|X|\le c$ \as{} for some
  constant $c$).
\end{example}

\begin{example}\label{E2}
  Let $X_n$ be independent random variables with
  $\P(X_n=n)=1/n$ and $\P(X_n=0)=1-1/n$, and take
  $a_n=1$. By the Borel--Cantelli lemma, 
$X_n=n$ infinitely often \as, and thus
  $\limsup_{\ntoo} X_n=\infty$ \as;
  consequently $X_n$ is not $O(1)$ a.s.
On the other hand, $X_n\pto0$, so $X_n=\op(1)$ and
  $X_n=O(1)$ \whp{} by \eqref{oooowhp}.
\end{example}

\begin{warning}
In particular, there is no analogue of \eqref{oooas} for
$O$ \as{} and $O$ w.h.p. 
Since ``\as'' usually is a strong notion compared to others (for
example for convergence), there is an obvious
risk of confusion and mistakes here, and it is important to be extra careful
when using ``$O(a_n)$ \as'' and ``$O(a_n)$ \whp''.  
\end{warning}

\section{A final warning}\label{Swarning}

Sometimes one sees expressions of the type $X_n=O(a_n)$ or
$X_n=o(a_n)$, for some random variables $X_n$,
without further qualifications or explanations. In analogy with \refS{Sowhp},
I think that the natural interpretations of these are the following:
\begin{defq}
\item\label{DOX}
$X_n=O(a_n)$ if there exists a constant
  $C$ such that $|X_n|\le Ca_n$ (surely, or  \as).
\item\label{DoX}
$X_n=o(a_n)$ if there exists a 
sequence $\gd_n\to0$ such  
that $|X_n|\le \gd_na_n$ (surely, or  \as).
\end{defq}
These notations are thus uniform estimates, and stronger than 
$X_n=O(a_n)$ \whp{} and \xnoanwhp, since no exceptional events of
small probabilities are allowed.

\begin{remark}\label{RSJW}
  As remarked in \refR{Rdiscrete}, in typical applications
  ``surely'' and ``a.s.'' are equivalent. When they are not, it is
  presumably best to follow standard probability theory practise and
  ignore events of probability 0, so the interpretation ``a.s'' in
  \ref{DOX}--\ref{DoX} seems best.
In this case, \ref{DOX}--\ref{DoX} are the same as
  \ref{DOloo}--\ref{Doloo}, so $X_n=O(a_n)\iff
  X_n=\Oloo(a_n)$ and $X_n=o(a_n)\iff X_n=\oloo(a_n)$.
\end{remark}

\begin{warning}
However, I guess that most times one of these notations is used,
\ref{DOX} or \ref{DoX} is
\emph{not} the intended meaning; either there are typos, or
the author really means something else, presumably one of the other
notions discussed above.   
\end{warning}

\begin{remark}
In the special situation that all $X_n$ are defined on a common
probability space as in \refS{SOas}, another reasonable
interpretation of $X_n=O(a_n)$ and $X_n=o(a_n)$ is
$X_n=O(a_n)$ \as{} and  $X_n=o(a_n)$
\as{}, see \ref{DOas}--\ref{Doas}. This is
equivalent to allowing random $C$ or $\gd_n$ in
\ref{DOX}--\ref{DoX}, and is a weaker property.
(This emphasizes the need for careful definitions to avoid ambiguities.)
\end{remark}

The notations \ref{DOX} and \ref{DoX} thus risk being ambiguous.
  If \ref{DOX} or \ref{DoX} really is intended, it
  may be better to use the unambiguous notation
  $\Olx\infty$ or $\olx\infty$, see \refS{Solp} and \refR{RSJW}.

%\begin{ack}
%These notes were written at Institut Mittag-Leffler,
%Djursholm, Sweden, during the program ``Discrete Probability'' 2009.
%I thank several participants for helpful comments and suggestions.
%\end{ack}

\newcommand\AAP{\emph{Adv. Appl. Probab.} }
\newcommand\JAP{\emph{J. Appl. Probab.} }
\newcommand\JAMS{\emph{J. \AMS} }
\newcommand\MAMS{\emph{Memoirs \AMS} }
\newcommand\PAMS{\emph{Proc. \AMS} }
\newcommand\TAMS{\emph{Trans. \AMS} }
\newcommand\AnnMS{\emph{Ann. Math. Statist.} }
\newcommand\AnnPr{\emph{Ann. Probab.} }
\newcommand\CPC{\emph{Combin. Probab. Comput.} }
\newcommand\JMAA{\emph{J. Math. Anal. Appl.} }
\newcommand\RSA{\emph{Random Struct. Alg.} }
\newcommand\ZW{\emph{Z. Wahrsch. Verw. Gebiete} }
\newcommand\DMTCS{\jour{Discr. Math. Theor. Comput. Sci.} }

\newcommand\AMS{Amer. Math. Soc.}
\newcommand\Springer{Springer-Verlag}
\newcommand\Wiley{Wiley}

\newcommand\vol{\textbf}
\newcommand\jour{\emph}
\newcommand\book{\emph}
\newcommand\inbook{\emph}
\def\no#1#2,{\unskip#2, no. #1,} %(typeset after year) 
\newcommand\toappear{\unskip, to appear}

\newcommand\webcite[1]{%\hfil  %???
   %\penalty0 %???
\texttt{\def~{{\tiny$\sim$}}#1}\hfill\hfill}
\newcommand\webcitesvante{\webcite{http://www.math.uu.se/~svante/papers/}}
\newcommand\arxiv[1]{\webcite{arXiv:#1.}}

\def\nobibitem#1\par{}

\end{document}